\documentclass{amsart}

\usepackage{amscd}
\usepackage{amssymb}

\usepackage[all]{xypic}
\usepackage[dvipsnames] {xcolor}
\usepackage{tikz}
\usepackage{soul}
\usepackage{marginnote}
\usepackage{enumerate}
\usepackage{scalefnt}

\numberwithin{equation}{section}

\theoremstyle{plain} 
\newtheorem{theorem}[equation]{Theorem}
\newtheorem{lemma}[equation]{Lemma}

\newtheorem{proposition}[equation]{Proposition}
\newtheorem{corollary}[equation]{Corollary}

\newtheorem*{theoremstar}{Theorem}

\theoremstyle{definition}
\newtheorem{definition}[equation]{Definition}

\newtheorem{example}[equation]{Example}

\newtheorem{remark}[equation]{Remark}
\newtheorem*{remark*}{Remark}

\newcommand{\defining}[1]{{\emph{#1}}}
\newcommand{\definedas}{:=}

\newcommand{\itemref}[1]{(\ref{#1})}

\newcommand{\reals}{{\mathbb{R}}}
\newcommand{\integers}{{\mathbb{Z}}}
\newcommand{\naturals}{{\mathbb{N}}}

\newcommand{\complexes}{{\mathbb{C}}}
\newcommand{\field}{{\mathbb{F}}}

\def\doCal#1{%
\ifx#1\doAllCalEnd\def\doAllCal{\relax}\else%
 \expandafter\edef\csname#1cal\endcsname{{\noexpand\mathcal #1}}\fi}
\def\doAllCal#1{\doCal#1\doAllCal}
\doAllCal ABCDEFGHIJHKLMNOPQRSTUVWXYZ\doAllCalEnd

\def\doBar#1{%
\ifx#1\doAllBarEnd\def\doAllBar{\relax}\else%
 \expandafter\edef\csname#1bar\endcsname{{\noexpand\overline{#1}}}\fi}
\def\doAllBar#1{\doBar#1\doAllBar}
\doAllBar ABCDEFGHIJHLMNOPQRSTUVWXYZabcdefghijhlmnopqrstuvwxyz\doAllBarEnd
\def\doWiggle#1{%
\ifx#1\doAllWiggleEnd\def\doAllWiggle{\relax}\else%
 \expandafter\edef\csname#1wiggle\endcsname{{\noexpand\tilde{#1}}}\fi}
\def\doAllWiggle#1{\doWiggle#1\doAllWiggle}
\doAllWiggle ABCDEFGHIJHLMNOPQRSTUVWXYZabcdefghijklmnopqrstuvwxyz\doAllWiggleEnd

\def\doBold#1{%
\ifx#1\doAllBoldEnd\def\doAllBold{\relax}\else%
 \expandafter\edef\csname#1bold\endcsname{{\noexpand\bf #1}}\fi}
\def\doAllBold#1{\doBold#1\doAllBold}
\doAllBold ABCDEFGHIJHKLMNOPQRSTUVWXYZabcdefghijklmnopqrstuvwxyz\doAllBoldEnd

\newcommand{\Largestrut}{\mbox{{\Large\strut}}}
\newcommand{\largestrut}{\mbox{{\large\strut}}}

\definecolor{llgray}{RGB}{230,230,230}
\newcommand{\highlight}[1]{\ifmmode{\text{\sethlcolor{llgray}\hl{$#1$}}}\else{\sethlcolor{llgray}\hl{#1}}\fi}
\makeatletter
\newcommand{\switchmarginpar}{
\if@reversemargin
\normalmarginpar
\else
\reversemarginpar
\fi
}
\makeatother

\DeclareMathOperator{\Aut}{Aut}

\DeclareMathOperator{\BSU}{BSU}
\DeclareMathOperator{\BSO}{BSO}

\newcommand{\Cen}{C}
\DeclareMathOperator{\Cat}{Cat}

\DeclareMathOperator{\Gr}{Gr}

\DeclareMathOperator{\Hom}{Hom}
\DeclareMathOperator{\Inj}{Inj}

\DeclareMathOperator{\id}{id}
\newcommand{\Id}{\mathrm{Id}}

\DeclareMathOperator{\Obj}{Obj}

\DeclareMathOperator{\Out}{Out}

\DeclareMathOperator{\SO}{SO}
\DeclareMathOperator{\Sp}{Sp}
\DeclareMathOperator{\SU}{SU}
\DeclareMathOperator{\U}{U}

\newcommand{\epi}{\twoheadrightarrow}

\newcommand{\pdash}{$p$\kern1.3pt-}
\newcommand{\twodash}{$2$\kern1.3pt-}
\newcommand{\whatever}{\text{\,--\,}}

\newcommand{\Zpinfinity}{\integers/{p}^{\infty}}
\newcommand{\ZpinfinitySpecific}[1]{\integers/{#1}^{\infty}}
\newcommand{\Linkwiggle}{\widetilde{\Lcal}}
\newcommand{\imatrix}{\ibold}
\newcommand{\jmatrix}{\jbold}
\newcommand{\kmatrix}{\kbold}
\newcommand{\imatrixbar}{{\overline{\imatrix}}}
\newcommand{\jmatrixbar}{{\overline{\jmatrix}}}
\newcommand{\kmatrixbar}{{\overline{\kmatrix}}}
\newcommand{\Pbullet}{P^{\bullet}}
\newcommand{\twobytwo}[4]{
   \left(\begin{array}{cc}
   #1 & #2\\
   #3 & #4
   \end{array}\right)}
\newcommand{\character}[1]{\chi_{#1}}

\newcommand{\Tcalbar}{\overline{\Tcal}}
\newcommand{\Fcalbar}{\overline{\Fcal}}
\newcommand{\Tboldbar}{\overline{\Tbold}}
\newcommand{\iboldbar}{\overline{\ibold}}

\newcommand{\subgroupeq}{\leq}
\newcommand{\subgroup}{<}

\newcommand{\supergroup}{>}
\newcommand{\Slominska}{S{\l}omi\'{n}ska}

\newcommand{\xlongrightarrow}[1]{\xrightarrow{\ #1\ }}

\newcommand\restr[2]{{
  \left.\kern-\nulldelimiterspace 
  #1 
  \vphantom{\big|} 
  \right|_{#2} 
  }}

\newcommand{\suchthat}[1]{\left|\, #1 \right. }

\newcommand{\realize}[1]{\left|#1 \right|}
\newcommand{\pcomplete}[1]{#1_{2}^{\wedge}}
\newcommand{\poddcomplete}[1]{#1_{p}^{\wedge}}

\usepackage{hyperref}
\hypersetup{colorlinks=true,linkcolor=violet,citecolor=purple} 

\definecolor{llgray}{RGB}{230,230,230}
\usepackage[nolabel]{showlabels} 

\begin{document}

\title{A new approach to mod $2$ decompositions\\
of $\BSU(2)$ and $\BSO(3)$}

\author{Eva Belmont}
\address{Department of Mathematics, Northwestern University, Evanston, IL 60208}
\email{ebelmont@northwestern.edu}

\author{Nat\`alia Castellana}
\address{Departament de Matem\`atiques, Universitat Aut\`onoma de Barcelona - BGSMath, Spain}
\email{natalia@mat.uab.cat}

\author{Jelena Grbi\'{c}}
\address{School of Mathematical Sciences, University of Southampton, Southampton, UK }
\email{j.grbic@soton.ac.uk}

\author{Kathryn Lesh}
\address{Department of Mathematics, Union College, Schenectady NY, USA}
\email{leshk@union.edu}

\author{Michelle Strumila}
\address{School of Mathematics and Statistics,
      Melbourne University, Parkville VIC, Australia}
\email{mstrumila@student.unimelb.edu.au}

\date{\today}

\begin{abstract}
Dwyer, Miller and Wilkerson proved that  at the prime~$2$, the classifying spaces of  $\SU(2)$ and $\SO(3)$ can be obtained as a homotopy pushout of the classifying spaces of certain subgroups. In this paper we show explicitly how these decompositions arise from the fusion systems of  $\SU(2)$ and~$\SO(3)$ over maximal discrete \twodash toral subgroups.
\end{abstract}

\maketitle
\markboth{\sc{Belmont, Castellana, Grbi\'{c}, Lesh, and Strumila}}
{\sc{Mod $2$ decompositions of $\BSU(2)$ and $\BSO(3)$}}


\section{Motivation}
In 1987, Dwyer, Miller and Wilkerson described a new homotopical construction of
$\BSU(2)$ completed at~$2$;
they used it to prove the homotopic uniqueness of $\pcomplete{\BSU(2)}$ as
a \twodash complete space whose cohomology is $\field_2[x_4]$ as an unstable algebra over the Steenrod algebra
(see \cite[Theorem 1.1]{DMW1}).

\begin{theoremstar}[{\cite[Theorem 4.1]{DMW1}}]
Let $N(T)\subgroupeq \SU(2)$ be the normalizer of the maximal torus, and let $Q_{16} \subgroupeq O_{48}\subgroupeq N(T)$ be the quaternionic subgroup
(of order~$16$) and the binary octahedral subgroup (of order~$48$), respectively.
Let $X$ be the homotopy pushout
\begin{equation}     \label{eq: DMW original}
\begin{gathered}
\xymatrix{
B\,Q_{16}
   \ar[d]\ar[r]
& BN(T)
   \ar[d]\\
B\,O_{48}
   \ar[r]
& X,
}
\end{gathered}
\end{equation}
where the maps are induced by inclusions. Then the induced map $X\to \BSU(2)$ is a homotopy equivalence after \twodash  completion.
\end{theoremstar}

Since $\SO(3)$ is a central quotient of $\SU(2)$,
a similar description of $\pcomplete{\BSO(3)}$ as a homotopy pushout can be derived from that of
$\pcomplete{\BSU(2)}$ (see \cite[Corollary 4.2]{DMW1}) and used to prove the same homotopic uniqueness with respect to the cohomology $\field_2[w_2,w_3]$ as an unstable algebra over the Steenrod algebra.
The pushout \eqref{eq: DMW original} was also studied in~\cite{JM-Arcata}, with broader methods, and interpreted it as an homotopy colimit over an orbit category with respect to certain \twodash toral subgroups of $\SU(2)$. They also applied their methods to another example, $\BSU(3)$ at $p=2$.

In 1990, M\"uller \cite{Muller} used the theorem above to give a complete description of the set of homotopy classes of maps from $\BSU(2)$ to~$BH$, where $H$ is one of
$\U(n)$, $\SU(n)$, $\SO(n)$, or~$\Sp(n)$. The result is in terms of the characters of the subgroups $Q_{16}$ and~$O_{48}$, and is one of the few examples in which a complete description of the set of homotopy classes of maps between classifying spaces of compact Lie groups $\left[BG,BH\right]$ is known.

Jackowski and McClure \cite{JM} proved that, up to \pdash completion, the classifying space of a compact Lie group $G$ can be obtained as a homotopy colimit of centralizers of elementary abelian \pdash subgroups.
Later Jackowski, McClure, and Oliver \cite{JMO} described $BG$ as a homotopy colimit  of classifying spaces of a family of \pdash toral subgroups, again up to \pdash completion. This decomposition is the key tool that allowed them to classify the homotopy classes of selfmaps of~$BG$. However, the earlier decomposition~\eqref{eq: DMW original} realised by Dwyer, Miller and Wilkerson did
not fit into these general methods, and was obtained ad hoc for the group~$\SU(2)$, and from there for~$\SO(3)$. The main results of our paper, Theorems~\ref{theorem:DMWSU(2)}
and~\ref{theorem:DMWSO(3)}, will show how the Dwyer-Miller-Wilkerson results can be reinterpreted in a discrete and combinatorial context in terms of normalizers of discrete \twodash toral subgroups, thereby situating it within a new general theoretical framework.

For finite groups~$G$, Dwyer \cite{Dwyer-Homology} unified the two previously known homology decompositions of $BG$ (\cite{JM}, \cite{JMO})
at the prime~$p$. The tool was the Borel construction for the conjugation action of~$G$ on a poset of subgroups of $G$ that is closed under conjugation.
The previous examples (\cite{JM}, \cite{JMO}) are obtained by using the poset
of nontrivial \pdash subgroups or nontrivial elementary abelian \pdash subgroups.
Dwyer formalized the notion of a homology decomposition of a topological space, and introduced three specific examples for classifying spaces of finite groups: the centralizer decomposition, the subgroup decomposition, and the normalizer decomposition.
We emphasize that the normalizer decomposition for finite groups was new in this context. It is indexed on a poset obtained from
the subdivision category of the poset of nontrivial finite \pdash groups;
the subgroups whose classifying spaces appear in the homology decomposition are then intersections of normalizers of \pdash subgroups. The existence of such a decomposition when $G$ is a compact Lie group is due to Libman \cite{Libman-Minami}, using techniques involving the action of~$G$ on the poset of nontrivial \pdash toral subgroups of~$G$.

Recently, the homotopy theory of saturated fusion systems on
discrete \pdash toral  groups has provided a new, more general framework for dealing with the homotopy type of \pdash completed classifying spaces of Lie groups
(\cite{BLO-Discrete}), in addition to other examples coming from finite loop spaces (\cite{BLO-LoopSpaces}). Saturated fusion systems
encode in a category the essential \pdash local information of the compact Lie group $G$ that is needed to uniquely determine the homotopy type of~$\poddcomplete{BG}$ (see \cite{Chermak}, \cite{Oliver-ExistenceL}, \cite{LL-ExistenceL}).
At first, this theory was developed by Broto, Levi and Oliver \cite{BLO-Finite} for saturated fusion systems over finite \pdash groups.
In this context, Libman \cite{Libman-normalizer} proved the existence of the normalizer decomposition for classifying spaces of saturated fusion systems over finite \pdash groups, extending the previous result of Dwyer.

In a forthcoming work, we will give a general setup for
a normalizer decomposition for classifying spaces of saturated fusions systems over discrete \pdash toral groups, extending the previous work of Libman for finite \pdash groups. We will also give new examples.
The main results of the current paper are
Theorems~\ref{theorem:DMWSU(2)} and~\ref{theorem:DMWSO(3)}, in which we explicitly show how the Dwyer-Miller-Wilkerson decompositions can be explained in terms of
the fusion systems of $\SU(2)$ and $\SO(3)$ at $p=2$ by considering the data that is used to encode the normalizer decomposition.

Sections~\ref{section: fusion systems} and~\ref{section: linking systems} describe the necessary background on the homotopy theory of fusion systems over discrete \pdash toral groups.
Section~\ref{section:2local} contains explicit computations of the \twodash local data of $\SU(2)$ required for proving the main result.
In Section~\ref{section: SU(2) SO(3)} we describe the nature of the
Dwyer-Miller-Wilkerson decompositions
of $\BSU(2)$ and $\BSO(3)$
as the normalizer decomposition with respect to the fusion systems of $\SU(2)$ and $\SO(3)$ at the prime~$2$.

\medskip

\subsection*{Acknowledgements.}
This paper is the first part of the authors' Women in Topology III project. A second part of that project will appear in a separate article. We are grateful to the organizers of the Women in Topology III workshop, as well as to the Hausdorff Research Institute for Mathematics, where the workshop beginning this research was held.
The Women in Topology III workshop was supported by NSF grant DMS-1901795, the AWM ADVANCE grant NSF HRD-1500481, and Foundation Compositio Mathematica. The second author was partially supported by FEDER-MEC grant MTM2016-80439-P.


\section{Fusion systems}
\label{section: fusion systems}

We start with a brief overview of fusion systems over discrete
\pdash toral groups, as given in the work of Broto, Levi and Oliver \cite{BLO-Discrete}.
A \defining{\pdash toral group} is an extension of a torus, $(S^1)^k$, by a finite \pdash group, and \pdash toral groups play the role for
compact Lie groups that is played by \pdash groups in the finite group setting. In order to make combinatorial models, however, one works with
a discrete version of \pdash toral groups.
The circle $S^1$ is replaced by its discrete \pdash analogue, $\Zpinfinity$,
the union of the cyclic \pdash groups $\integers/p^n$ under the standard inclusions.

\begin{definition}
A \defining{discrete \pdash toral group} is a group $P$ given by an extension
\[
1\longrightarrow \left(\Zpinfinity\right)^k \longrightarrow P\longrightarrow \pi_{0}P\longrightarrow 1,
\]
where $k$ is a nonnegative integer and $\pi_{0}P$ is a finite \pdash group.
We define the \defining{identity component} $P_{0}$ of $P$ by
$P_{0}\definedas \left(\Zpinfinity\right)^k$.
We call $\pi_{0}P$ the \defining{set of components} of~$P$.
\end{definition}

Given discrete \pdash toral groups $P$ and $Q$, let $\Hom (P,Q)$
denote the set of group homomorphisms from $P$ to $Q$, and let $\Inj
(P,Q)$ denote the set of group monomorphisms. If $P$ and $Q$ are subgroups
of a larger group~$S$, then we write $\Hom _S(P,Q)$ for
the set of  those homomorphisms induced by conjugation by
an element of~$S$. The following definition gives a tool to encode information on conjugacy relations among discrete \pdash toral subgroups of a compact Lie group.

\begin{definition}\cite[Definition 2.1]{BLO-Discrete}\label{def:fusionsystem}
A \defining{fusion system} $\Fcal$ over a discrete \pdash toral group $S$ is a
subcategory of the category of groups, defined as follows. The objects of
$\Fcal$ are all of the subgroups of~$S$. The morphism sets $\Hom_{\Fcal}(P,Q)$
contain only group monomorphisms, and satisfy the following conditions.
\begin{enumerate}[(i)]
\item
\label{item: conjugation condition}
$\Hom_S(P,Q)\subseteq \Hom_{\Fcal}(P,Q)$ for all $ P,Q\subgroupeq  S $.
In particular, conjugation by the identity element, i.e. subgroup inclusions, are in~$\Fcal$.
\item
\label{item: inclusion condition}
Every morphism in $\Fcal$ factors as the composite of an isomorphism in $\Fcal$ followed by a subgroup inclusion.
\end{enumerate}
Two subgroups $P,P' \subgroupeq  S $
are called \defining{$\Fcal$-conjugate} if they are isomorphic as objects of~$\Fcal$.
\end{definition}

For a subgroup $P\in\Obj(\Fcal)$, let $ \Aut_{\Fcal}(P) = \Hom_{\Fcal}(P,P)$ be the set of all morphisms in $\Fcal$ from $P$ to itself. We write $\Aut_{S}(P)\subgroupeq \Aut_{\Fcal}(P)$ for automorphisms of $P$ induced by conjugation by elements of~$S$.
Note that $\Aut_{P}(P)\cong P/Z(P)$ is the inner automorphism group of~$P$ (where $Z(P)$ denotes the center of~$P$). Just as the group-theoretical outer automorphism group of $P$ is defined as $\Out(P)\definedas \Aut(P)/\Aut_P(P)$, we define
$\Out_{\Fcal}(P) \definedas \Aut_{\Fcal}(P)/\Aut_{P}(P)$. Likewise, to restrict attention to outer automorphisms that are induced by conjugation in the supergroup~$S$, we write $\Out_{S}(P)=\Aut_{S}(P)/\Aut_{P}(P)$.

\begin{example}\label{ex:fusionG}
Let $G$ be a finite group, and let $S$ be a Sylow \pdash subgroup of~$G$. With these data we can define a fusion system denoted~$\Fcal_S(G)$. The objects of $\Fcal_S(G)$ are subgroups $P\subgroupeq  S$,
and morphisms from $P$ to $Q$ are group homomorphisms induced by conjugation in~$G$;
that is,  $\Hom_{\Fcal_S(G)}(P,Q)\definedas \Hom_{G}(P,Q)$.
The category $\Fcal_S(G)$ satisfies conditions \itemref{item: conjugation condition} and \itemref{item: inclusion condition} in Definition~\ref{def:fusionsystem}.
\end{example}

However, the definition of a fusion system includes examples that are too general for the aim of developing the homotopy theory of classifying spaces.
For example, one can take the fusion system consisting of all subgroups of~$S$, with morphisms given by all group monomorphisms.
Puig identified the two  key properties of  Example~\ref{ex:fusionG} that need to be abstractly formalized: the \pdash group $S$ should play the role of a  maximal \pdash subgroup, and  morphisms should behave like conjugations.
Puig introduced such axioms in a new definition of a \defining{saturated fusion system} over a finite \pdash group in \cite{Puig} as an abstract model of the \pdash local information about a finite group. The technical details can also be found in \cite{BLO-Finite}.

Aiming to describe an algebraic model for classifying spaces of compact Lie groups, Broto, Levi and Oliver first generalized the definition of a fusion system by considering discrete \pdash toral groups instead of finite \pdash groups \cite[Definition~2.1]{BLO-Discrete}. The starting point is the theory of maximal tori of Lie groups and their normalizers.
A~compact Lie group $G$ has a maximal torus, denoted~$T$, which is unique up to conjugacy and is a maximal connected abelian subgroup.
The Weyl group $W_G(T)\definedas N_G(T)/T$ is
a finite group, and we fix a Sylow \pdash subgroup $W_{p}\subgroupeq  W_G(T)$.
Let $N_p$ denote the inverse image of~$W_{p}$ in~$N_G(T)$.
The action of $W$ on $T$ restricts to an action of~$W_p$, and we have $N_p$
as the corresponding extension
\[
1\longrightarrow T\longrightarrow N_{p}\longrightarrow W_{p}\longrightarrow 1.
\]
Then $N_p$ is a maximal \pdash toral subgroup of~$G$ and is unique up to conjugacy in~$G$.

Any \pdash toral group has a dense discrete \pdash toral subgroup such that the inclusion morphism induces a mod $p$ equivalence on classifying spaces (see \cite[Corollary~1.2, Proposition~2.3]{Feshbach}). Let $S\subgroupeq  N_{p}$ be such a dense discrete \pdash toral subgroup. The group $S$ is a maximal discrete \pdash toral subgroup  of~$G$;  that is, any other  discrete \pdash toral subgroup of $G$ is conjugate by an element of~$G$ to a subgroup of~$S$.
More details be found in the proof of Proposition~9.3 in \cite{BLO-Discrete}.

\begin{definition}\label{definition:fusionGLie}
Let $G$ be a compact Lie group with maximal discrete \pdash toral subgroup~$S$. The \defining{fusion system of~$G$}, denoted $\Fcal _S(G)$, has as its object set all subgroups of $S$, and for $ P,Q \subgroupeq  S$ the morphisms are
$\Hom _{\Fcal_S(G)}(P,Q) \definedas \Hom _G(P,Q)$.
\end{definition}

Building on the categorical setup for finite groups, Broto, Levi, and Oliver
\cite[Definition~2.2]{BLO-Discrete} defined a \defining{saturated fusion system} over a discrete \pdash toral group~$S$. It has the same favorable characteristics as those over finite groups, namely that $S$ plays the role of a Sylow \pdash subgroup and that morphisms behave like conjugations.  The following proposition tells us that the fusion system of a compact Lie group has the desired properties.

\begin{proposition}\cite[Proposition~8.3]{BLO-Discrete}
\label{proposition:out-finite}
If $G$ is a compact Lie group with maximal discrete \pdash toral subgroup~$S$,
then the fusion system $\Fcal _S(G)$ is saturated.
\end{proposition}

Saturated fusion systems over finite \pdash groups are finite categories, but in general a saturated fusion system over a discrete \pdash toral group might be very large. Without making any further restrictions,
one would obtain
a description of $\poddcomplete{BG}$ as a homotopy colimit over a large category.
The strategy is to locate smaller subcategories of saturated fusion systems,
with finitely many isomorphism classes of objects, which still control the same \pdash local information.
For this reason, Broto, Levi, and Oliver introduced the \defining{bullet construction}, $(\whatever)^\bullet\colon \Fcal \longrightarrow\Fcal$,
whose image is a full subcategory of $\Fcal$ that has
a finite number of isomorphism classes of objects, but retains the same \pdash local information as~$\Fcal$.

The bullet construction is somewhat involved.
Let $S$ be a discrete \pdash toral group, with identity component $T=S_{0}$,
and let $W=\Aut_\Fcal(T)$. For $P\subgroupeq  T$, define $C_W(P)$ to be the group consisting of all $w\in W$ such that $\restr{w}{P}=\id_P$. Further, for $P\subgroupeq  T$ we set $I(P)=T^{C_W(P)}\subgroupeq  T$, that is, the subgroup consisting of all elements of the torus that are fixed by elements of $W$ that also fix $P$ pointwise.

\begin{proposition}
\cite[Definition~3.1, Lemma~3.2, Proposition~3.3]{BLO-Discrete}
\label{proposition: bullet}
Let $\Fcal$ be a saturated fusion system over a discrete \pdash toral group~$S$ with
identity component $T=S_{0}$, and let $p^m$ be the exponent of~$S/T$.
For each $P\subgroupeq  S$, consider
\[
P^{[m]}=
\langle g^{p^m}\suchthat{\largestrut g\in P}
\rangle\subgroupeq  P\cap T.
\]
There is an idempotent endofunctor $(\whatever)^\bullet\colon \Fcal \longrightarrow\Fcal$, the \defining{bullet functor}, given by
\[
P^\bullet=P\cdot I(P^{[m]})_0
   =\left\{gt\suchthat{g\in P,\  t\in I(P^{[m]})_0}\right\}.
\]
The full subcategory $\Fcal^{\bullet}\subseteq\Fcal$ with
$\Obj(\Fcal^{\bullet})\definedas \left\{P^\bullet \suchthat{P\subgroupeq  S}\right\}$
is closed under $\Fcal$-conjugacy and contains finitely many
$S$-conjugacy classes.
\end{proposition}

\begin{example}
Let $\Fcal$ be a saturated fusion system over a discrete \pdash torus~$T$. In this situation, the group of components of $T$ is trivial and $m=0$.  For all $P\subgroupeq  T$, we have $P^{\bullet}=I(P)_{0}$ since $P\leq I(P)_0$.
\end{example}

We restrict even further to a full subcategory of~$\Fcal^\bullet$
that will play an important role in the proof of Theorem~\ref{theorem:DMWSU(2)}.
Let $\Fcal$ be a saturated fusion system over a discrete \pdash toral group~$S$. Broto, Levi and
Oliver \cite[Proposition~2.3]{BLO-Discrete}
showed that $\Out_\Fcal(P)\definedas\Aut_{\Fcal}(P)/\Aut_{P}(P)$ is finite for all $P\subgroupeq S$.

\begin{definition}\label{definition:Fcentric-Fradical}
Let $\Fcal$ be a saturated fusion system over a discrete \pdash toral group~$S$.
\begin{enumerate}[(i)]
\item \label{item: discrete Fcentric}
    A~subgroup $P\subgroupeq  S$ is called \defining{$\Fcal$-centric}
    if $P$ contains all elements of $S$ that centralize it,
   and likewise all $\Fcal$-conjugates of $P$ contain their
    $S$-centralizers.
\item \label{item: discrete Fradical}
    A subgroup $P\subgroupeq  S$ is called \defining{$\Fcal$-radical}
    if $\Out_{\Fcal}(P)$ contains no nontrivial normal \pdash subgroup.
\end{enumerate}
We write $\Fcal^{c}$ (respectively, $\Fcal^{cr}$) for the full subcategory of $\Fcal $ whose objects are all subgroups of $S$ that are $\Fcal$-centric (respectively, both $\Fcal$-centric and $\Fcal$-radical).
\end{definition}

\begin{remark}  \label{remark: S is radical}
It is a consequence of the saturation axioms that $S$ itself is $\Fcal$-radical. From \cite[Definition 2.2 (i)]{BLO-Discrete}, we have that $\Out_\Fcal(S)$ is finite and of order prime to~$p$.
\end{remark}

We will see in Section~\ref{section: linking systems} that only
the subcategory $\Fcal^{cr}\subseteq\Fcal$ is needed in order to determine a \pdash completed classifying space.
Because of the following proposition, the bullet construction
allows us to narrow the search in specific instances. The approach is exemplified in Section~\ref{section:2local}.

\begin{proposition}
\cite[Proposition 3.3, Corollary~3.5]{BLO-Discrete}
\label{proposition:FiniteCentricRadicals}
Let $\Fcal$ be a saturated fusion system over a discrete \pdash toral group $S$. All $\Fcal$-centric $\Fcal$-radical subgroups of~$S$
are in~$\Obj(\Fcal^{\bullet})$. In particular, there are only finitely many conjugacy classes of such subgroups.
\end{proposition}


\section{The classifying space of a fusion system}
\label{section: linking systems}

In \cite{BLO-FiniteGroups}, Broto, Levi, and Oliver introduced fusion systems in homotopy theory in their program to give an algebraic description of self-homotopy equivalences of $\poddcomplete{BG}$
for finite groups~$G$. One then wants to know to what extent the actual homotopy type of $\poddcomplete{BG}$ is determined by algebraic information
encoded in~$\Fcal_S(G)$. Assume that $S\subgroupeq G$
is a Sylow \pdash subgroup of~$G$.
Oliver proved in \cite{Oliver-ExistenceL} that the homotopy type of $\poddcomplete{BG}$ is, in fact, completely determined by~$\Fcal_S(G)$:
given two finite groups $G$ and $H$, there is an equivalence $\poddcomplete{BG}\simeq \poddcomplete{BH}$ if and only if the
associated fusion systems are equivalent.

One can ask, if $\Fcal$ is an abstract saturated fusion system, not known to be defined as the fusion system $\Fcal_S(G)$ of a group~$G$, what topological space plays the role of the associated \pdash completed classifying space?
The nerve of the category $\Fcal$ itself is not a candidate, because it does not give the right answer when the fusion system \emph{does} come from a group. In particular, the center of~$G$, whose conjugations give trivial automorphisms,
is not detected by~$\Fcal_{G}(S)$. For example, if $G=\integers/p$ then $\realize{\Fcal_{G}(S)}$ is contractible. Instead, the classifying space of a saturated fusion system~$\Fcal$ will be the \pdash completed nerve of an abstract category  introduced by Broto, Levi, and Oliver \cite{BLO-Finite}, namely the centric linking system associated to~$\Fcal$, which is set up to recover the ``missing centers."

In order to motivate the abstract definition of a linking system, we
first introduce the notion of a transporter category. For a pair of subgroups $P,Q\subgroupeq  G$, let
$N_{G}(P,Q)=\left\{ g\in G \suchthat{gPg^{-1}\subgroupeq  Q} \right\}$.

\begin{definition}\label{definition:transporter}
If $G$ is a group and $\Hcal$ is a set of subgroups of~$G$, we define
the \defining{transporter category} for~$\Hcal$, denoted $\Tcal_\Hcal(G)$, as the category
whose object set is~$\Hcal$, and whose morphism sets are defined by
\[
\Hom_{\Tcal_\Hcal(G)}(P,Q)=\left\{ g\in G \suchthat{gPg^{-1}\subgroupeq  Q}
                          \right\}=N_{G}(P,Q).
\]
In particular, if $P\in\Hcal$, then the automorphism group $\Aut_{\Tcal_\Hcal(G)}(P)$ is given by~$N_{G}(P)$.
\end{definition}

\begin{example}\label{example: transporter for finite group}
Let $G$ be a finite group, and let $S$ be a Sylow \pdash subgroup of~$G$. Let~$\Hcal$
be the set of all non-trivial subgroups of~$S$.
By \cite[Lemma~1.2]{BLO-FiniteGroups}, there is a functor $\Tcal_\Hcal(G)\rightarrow \Bcal G$, where $\Bcal G$ is the category with a single object and $G$ as morphisms, that induces a homotopy equivalence
\[
\poddcomplete{\realize{\Tcal_\Hcal(G)}}
     \xrightarrow{\,\simeq\,} \poddcomplete{BG}
\]
if $\Hcal$ is an ``ample" family in the sense of Dwyer (see \cite{Dwyer-Homology}).\footnote{The authors used the notation $\Linkwiggle$ for the transporter category.}
The statement  uses Dwyer's results on homology decompositions. Indeed, \cite{Dwyer-Homology}
gives multiple families $\Hcal$ of subgroups possessing the feature that the transporter category for $\Hcal$ recovers the homotopy type of~$\poddcomplete{BG}$.
\end{example}

Example~\ref{example: transporter for finite group} suggests the likely usefulness of
a category that mimics a transporter category for an abstract setting.
In particular, \cite{BLO-Finite} and \cite{BLO-LoopSpaces} introduced the
notion of a centric linking system to provide an appropriate analogue
for the transporter category,
 and from there for the classifying space of a group. Recall that $\Fcal^{c}$ denotes the full subcategory of a fusion system $\Fcal$ whose objects are $\Fcal$-centric (Definition~\ref{definition:Fcentric-Fradical}).

\begin{definition}\cite[Definition 4.1]{BLO-Discrete}
\label{definition: linking}
Let $\Fcal $ be a fusion system over the discrete \pdash toral group~$S$.  A centric linking system associated to~$\Fcal$ is a category $\Lcal$
whose objects are the $\Fcal$-centric subgroups of $S$, together with a functor
\[
\Lcal \xlongrightarrow{\pi} \Fcal^c
\]
and distinguished monomorphisms
$\delta_P\colon P\rightarrow \Aut_\Lcal(P)$ for each $\Fcal$-centric subgroup $P\subgroupeq  S$, that satisfy the following conditions.

\begin{enumerate}[(i)]
\item The functor $\pi$ is the identity on objects, and surjective on morphisms. More precisely, for each pair of objects $ P,Q \in \Lcal$,
the center $Z(P)$ acts freely on $\Hom _{\Lcal}(P,Q)$ by precomposition, and
$\pi_{P,Q}$ induces a bijection
\[
\Hom _{\Lcal}(P,Q)/Z(P)\xrightarrow{\ \cong\ } \Hom _{\Fcal}(P,Q).
\]

\item For each $\Fcal$-centric subgroup $P\subgroupeq  S$ and each $ g \in P$,
the functor $\pi$ sends the element
$\delta _{P}(g) \in \Aut_{\Lcal}(P)$ to $c_g \in \Aut _{\Fcal}(P)$.

\item For each $ f \in \Hom _{\Lcal}(P,Q) $ and each $g \in P$, the following square in~$\Lcal$ commutes:
\[
\diagram
P \rto^{f} \dto_{\delta _{P}(g)}
      & Q \dto^{\delta_{Q}(\pi(f)(g))}
\cr P \rto^f & Q.
\enddiagram
\]
\end{enumerate}
\end{definition}

The existence and uniqueness, up to equivalence,
of centric linking systems associated to a given saturated fusion system is a key result in the theory.  For saturated fusion systems over a finite \pdash group~$S$,
uniqueness was proven first in \cite{Chermak} by introducing the new theory of localities.
Another proof was given by Oliver in \cite{Oliver-ExistenceL} using the obstruction theory developed in \cite{BLO-Finite}. Later,  \cite{LL-ExistenceL} extended the result to saturated fusion systems over discrete \pdash toral groups.

\begin{theorem}\cite{LL-ExistenceL}
\label{theorem: LL-ExistenceL}
Let $\Fcal$ be a saturated fusion system over a discrete \pdash toral group. Up to equivalence, there exists a unique centric linking system associated to~$\Fcal$.
\end{theorem}

As a consequence of this uniqueness, we can make the following definition of the classifying space of a saturated fusion system, without explicit reference to the linking system being used.

\begin{definition}
\label{definition:classifyingspacefusion}
Let $\Fcal$ be a saturated fusion system over a discrete \pdash toral group~$S$.
The \defining{classifying space of~$\Fcal$} is defined as $B\Fcal\definedas\poddcomplete{\realize{\Lcal}}$, where $\Lcal$ is a centric linking system associated to~$\Fcal$.
\end{definition}

In \cite[Sections~9 and~10]{BLO-Discrete}, the authors show how the \pdash completions of classifying spaces of compact Lie groups and of \pdash compact groups (as defined by Dwyer and Wilkerson in \cite{dwyer-wilkerson-fixed-point}) can be described as classifying spaces of associated fusion systems. Later, \cite{BLO-LoopSpaces} showed that the \pdash completion of the classifying space of a finite loop space can also be described as the classifying space of a saturated fusion system
at the prime~$p$. Other examples (see \cite[Section~8]{BLO-Discrete}) come from linear torsion groups.

In the case of the fusion system associated to a compact Lie group,
Broto, Levi, and Oliver \cite[Proposition 9.12]{BLO-Discrete}
explicitly construct a centric linking system $\Lcal_S(G)$ associated to $\Fcal_S(G)$
(see Definition~\ref{definition:fusionGLie}).
This construction uses the group structure of $G$, and recovers the homotopy type of $BG$ up to \pdash completion.

\begin{theorem}
\cite[Theorem~9.10]{BLO-Discrete}
\label{theorem:Lpcom=BGpcom}
Let $G$ be a compact Lie group, and fix a maximal discrete \pdash toral subgroup $S\subgroupeq  G$.  There exists a centric linking system $\Lcal_S(G)$ associated to $\Fcal_S(G)$ such that $\poddcomplete{\realize{\Lcal_S(G)}}\simeq \poddcomplete{BG}$.
\end{theorem}

The objects of a centric linking system associated to a fusion
system $\Fcal$ over $S$ are the $\Fcal$-centric subgroups of~$S$
(see Definition~\ref{definition: linking}).
It is helpful for computational purposes to reduce the set of objects from the set of all $\Fcal$-centric subgroups to a smaller number, in order
to obtain a more manageable category.
Let $\Lcal$ be a centric linking system associated to a saturated fusion system $\Fcal$ over a discrete \pdash toral group~$S$.
Define full subcategories $\Lcal^{cr}\subseteq \Lcal^{c\bullet}\subseteq \Lcal$, with $\Obj(\Lcal^{cr})$ given by the subgroups $P\subgroupeq  S$ that are both
$\Fcal$-centric and $\Fcal$-radical, and
$\Obj(\Lcal^{c\bullet})=\Obj(\Fcal^c)\cap \Obj(\Fcal^\bullet)$.

\begin{proposition}\label{proposition:Hfamilies}
Let $\Fcal$ be a saturated fusion system over a discrete \pdash toral group~$S$. If $\Lcal$ is a centric linking system
associated to $\Fcal$, then the inclusions of full subcategories $\Lcal^{cr}\subseteq \Lcal^{c\bullet}\subseteq \Lcal$ induce homotopy equivalences
$\poddcomplete{\realize{\Lcal^{cr}}}
     \simeq \poddcomplete{\realize{\Lcal^{c\bullet}}}
     \simeq \poddcomplete{\realize{\Lcal}}$.
\end{proposition}

\begin{proof}
Given a centric linking system $\Lcal$ associated to~$\Fcal$,
it is shown in \cite[Proposition 4.5]{BLO-Discrete}
that restriction to the full subcategory $\Lcal^{c\bullet}$
induces a homotopy equivalence of nerves
$\realize{\Lcal^{c\bullet}}\simeq \realize{\Lcal}$.
Finally \cite[Corollary~A.10]{BLO-LoopSpaces} shows that the inclusion $\Lcal^{cr}\subset \Lcal$ induces an equivalence on \pdash completed nerves.
\end{proof}

\begin{remark*}\label{reamark:finiteHfamilies}
If $\Fcal$ is a saturated fusion system over a \emph{finite} \pdash group, then for any $P\subgroupeq S$ we have $P^\bullet=P$, and in this case Proposition~\ref{proposition:Hfamilies} was proved in \cite[Theorem~3.5]{BCGLO}.
\end{remark*}

In general, a centric linking system for a group $G$ is constructed as a subquotient of the transporter category~$\Tcal_S(G)$ (Definition~\ref{definition:transporter}).
We use the explicit construction in \cite[Proposition 9.12]{BLO-Discrete} to prove the
proposition below. The advantage
for explicit computations is that it allows a description of the classifying space directly in terms of a small subcategory of the transporter category.

Let $\Tcal^{cr}_S(G)$ denote the full subcategory of the transporter category $\Tcal_S(G)$ whose objects are $\Fcal_S(G)$-centric $\Fcal_S(G)$-radical subgroups.

\begin{proposition}\label{proposition:L=T}
Let $G$ be a compact Lie group.
Assume that for all $\Fcal_S(G)$-centric $\Fcal_S(G)$-radical subgroups $P\subgroupeq S$, we know that $C_{G}(P)$ is a finite \pdash group and also that $C_G(P)=Z(P)$.
Then there is a functor $\sbar\colon \Lcal_S^{cr}(G)\rightarrow \Tcal_S^{cr}(G)$ that is an isomorphism of categories.
\end{proposition}

\begin{proof}
We refer to the explicit construction of the centric linking system in \cite[Proposition 9.12]{BLO-Discrete}. Under our hypothesis, when restricted to $\Fcal$-centric $\Fcal$-radical discrete \pdash toral subgroups,
one can take $s=\id$ in the pullback diagram
\[
\diagram
\Lcal_S^{cr}(G) \rto^-{\sbar} \dto_{ }
      & \Tcal_S^{cr}(G) \dto^{ }
\cr \Fcal^{cr}_S(G) \rto^-{\id} & \Tcal^{cr}_S(G)/\Zcal,
\enddiagram
\]
where $\Zcal(P)=Z(P)$. This is because $C_G(P)=Z(P)$ is a finite \pdash group and $C_G(P)$ does not contain elements of finite order prime to~$p$.
Then $\sbar$ gives the desired isomorphism of categories.
\end{proof}

Lastly, we tie back directly to the goal of obtaining information about~$\poddcomplete{BG}$. Recall that $\Bcal G$ is the topological category with one object and the group $G$ as morphisms. There is a functor $ \Tcal_S(G) \rightarrow \Bcal G$ defined by sending $g\in N_G(P,Q)$ to $g\in G$ for all $P,Q\subgroupeq G$.

\begin{corollary}\label{corollary:T=BG}
Under the hypothesis of Proposition~\ref{proposition:L=T}, the functor
$ \Tcal_S(G) \rightarrow \Bcal G$ induces an equivalence of \pdash completed nerves,
$\poddcomplete{\realize{\Tcal^{cr}_S(G)}}\simeq \poddcomplete{BG}$.
\end{corollary}

\begin{proof}
The corollary is a consequence of Proposition~\ref{proposition:Hfamilies}, Proposition~\ref{proposition:L=T},
and the proof of \cite[Theorem~9.10]{BLO-Discrete}.
\end{proof}


\section{The fusion system of $SU(2)$ at $p=2$}
\label{section:2local}

Throughout this section, let $G=\SU(2)$, and let $p=2$. In the first part of the section, we compute the category $\Fcal^{cr}_S(G)$, that is the full subcategory of $\Fcal_S(G)$ (Definition~\ref{definition:fusionGLie})
consisting of $\Fcal_S(G)$-centric, $\Fcal_S(G)$-radical subgroups (Definition~\ref{definition:Fcentric-Fradical}). We compute from first principles to illustrate this example.
In the second part of the section, we compute the associated linking system, again from first principles, using
Proposition~\ref{proposition:L=T}.
We will use the computation in Section~\ref{section: SU(2) SO(3)}, to interpret the
mod~$2$ homotopy decompositions of $\BSU(2)$ and $\BSO(3)$ due to Dwyer, Miller and Wilkerson~\cite{DMW1} as ``normalizer decompositions."

\smallskip
We set up notation for the matrix
representation we use. Elements of $\SU(2)$ are matrices of the form
$
\twobytwo{a}{b}{-\bbar}{\abar}
$
where $a,b\in\complexes$ and $|a|^2+|b|^2=1$.
The maximal torus has rank~$1$, and can be represented as
\begin{equation}     \label{eq: torus presentation SU}
\Tbold =\left\{\twobytwo{a}{0}{0}{\abar}\suchthat{\Largestrut a\in S^1}
        \right\}
       \cong S^{1}.
\end{equation}

In order to work with the fusion system associated to $\SU(2)$ as in
Definition~\ref{definition:fusionGLie}, we need to describe a maximal discrete
\twodash toral subgroup of $\SU(2)$. The normalizer of the maximal torus, which can be found by direct computation with matrices, is generated by the torus itself, together with one additional element~$\imatrix$ of order~$4$:
\begin{equation}   \label{eq: normalizer in SU(2)}
\Nbold_{\SU(2)}\Tbold=\left\langle\, \strut \Tbold, \imatrix\,\right\rangle,
    \mbox{\quad where $\imatrix=\twobytwo{0}{1}{-1}{0}.$}
\end{equation}
The element $\imatrix$ acts on $\Tbold$
by complex conjugation of the entry~$a$ in~\eqref{eq: torus presentation SU}. There is an extension
\begin{equation}   \label{eq: extension for topological}
1\longrightarrow \Tbold \longrightarrow \Nbold_{\SU(2)}\Tbold
        \longrightarrow  \integers/2\longrightarrow 1.
\end{equation}
We conclude from this calculation that the full normalizer of the maximal torus of
$\SU(2)$ is already a \twodash toral group, with no need to look for a maximal
\twodash toral subgroup. Note that the quotient in \eqref{eq: extension for topological}
is represented by~$\ibold$,
and the extension is not split,
since all lifts of the nontrivial element of $\integers/2$
to $\Nbold_{\SU(2)}\Tbold$ have order~$4$.

Next we discretize. Let $T=\ZpinfinitySpecific{2}\subgroupeq\Tbold$ be the set of \twodash torsion elements, which is dense in~$\Tbold$. The group
$S\definedas \left\langle\, \strut T, \imatrix\,\right\rangle
\subgroupeq  \Nbold_{\SU(2)}\Tbold$,
is a maximal discrete \twodash toral subgroup of~$\SU(2)$, given by
the extension
\begin{equation}   \label{eq: ses for S}
1\longrightarrow \ZpinfinitySpecific{2} \longrightarrow S
        \longrightarrow  \integers/2\longrightarrow 1,
\end{equation}
where again $\integers/2$ is generated by the coset of~$\imatrix$.
Our goal is to compute a fusion system over $S$ that is adequate to yield the \twodash completion of~$\BSU(2)$. For the rest of this section, let
$\Fcal \definedas \Fcal_S(SU(2))$ be the fusion system over $S$ given by Definition~\ref{definition:fusionGLie}, with objects given by all
subgroups $P$ of~$S$ and morphisms given by conjugations in~$\SU(2)$. In order to use Proposition~\ref{proposition:Hfamilies},
we would like to describe the $\Fcal$-centric, $\Fcal$-radical subgroups of~$S$.

To narrow down which subgroups of $S$ could be $\Fcal$-centric and $\Fcal$-radical, we use
the bullet functor from Proposition~\ref{proposition: bullet} and
compute the collection
$\Fcal^{\bullet}\definedas \{ P^\bullet \suchthat{P\subgroupeq  S}\}$. This collection contains all the $\Fcal$-centric, $\Fcal$-radical subgroups of~$S$
(Proposition~\ref{proposition:FiniteCentricRadicals}), though it may also contain other groups.
There are two possibilities for subgroups $P\subgroupeq S$:
either $P\subgroupeq T=\ZpinfinitySpecific{2}$,
or $P$ is generated by a subgroup of $T$ together with~$\imatrix$.
It turns out that putting in enough elements of~$\ZpinfinitySpecific{2}$ causes
the bullet construction to put in the rest, as indicated in the lemma below.
Recall that for any $P\subgroupeq  S$, the notation $\Aut_\Fcal(P)$ denotes the group of automorphisms of $P$ in the fusion category, which in this case means those induced by conjugation in~$\SU(2)$. Then $W=\Aut_\Fcal(T)\cong\integers/2$.

\begin{lemma}    \label{lemma: SU2 bullets}
If $P$ contains the subgroup $\integers/8\subgroup T$, then $P^{\bullet}$ contains~$T$.
\end{lemma}

\begin{proof}
We follow the construction in Proposition~\ref{proposition: bullet} and compute the identity component of~$P^{\bullet}$. The order of
the Weyl group $W=\Aut_\Fcal(T)$ is~$2^{1}$,
so we first look at~$P^{[1]}$, the subgroup of $P$ generated by squares of elements. Since $\integers/8\subgroupeq P$, we find that $\integers/4\subgroupeq P^{[1]}$.
The group $W\cong\integers/2$ acts on $T$ by complex conjugation, so it has the nontrivial action on~$\integers/4$.
As a result, only the identity element of $W$ centralizes~$P^{[1]}$, and the fixed set of $\id\in W$ is all of~$T$. Hence $\Pbullet=P\cdot T$ contains~$T$.
\end{proof}

As the quaternion group will appear shortly, we set corresponding notation:
let $\jmatrix=\twobytwo{i}{0}{0}{-i}$, where $i=\sqrt{-1}$. Observe that $\jmatrix$ has order~$4$, and
$\jmatrix\in T$ by taking $a=i$ in~\eqref{eq: torus presentation SU}.

\begin{lemma}    \label{lemma: radical possibilities}
If $P$ is a proper $\Fcal$-centric and $\Fcal$-radical subgroup of~$S$, then
$P$ is conjugate in $S$ to one of the following groups, where the left side of each extension is contained in the discrete torus~$T=\ZpinfinitySpecific{2}$
and the right side is represented by~$\imatrix$.
\begin{enumerate}[(i)]
\item \label{item: nonsurviving option}
$P=\langle\,\imatrix\,\rangle$, which is the nontrivial extension
\[
1\rightarrow\integers/2\rightarrow P\rightarrow\integers/2\rightarrow 1.
\]
\item \label{item: Q extension}
$P=\left\langle\,\jmatrix, \,\imatrix\, \right\rangle$, which is an extension
\[
1\rightarrow\integers/4\rightarrow P\rightarrow\integers/2\rightarrow 1.
\]
\end{enumerate}
\end{lemma}

\begin{proof}
If $P$ is a proper subgroup of~$T$, then
it is not $\Fcal$-centric, because $T$ centralizes~$P$. If $P=T$,
then $C_{S}(T)=T$, so $T$ is $\Fcal$-centric. However, $T$ is not $\Fcal$-radical (Definition~\ref{definition:Fcentric-Fradical}\itemref{item: discrete Fradical}):
automorphisms of $T$ are given by conjugations in~$S$, so
\[
\Out_\Fcal(T)=\Aut_{\Fcal}(T)/\Aut_{T}(T)=\integers/2,
\]
generated by $\imatrix$ acting on $T$ by complex
conjugation. Since $\Out_\Fcal(T)$ is itself a \twodash group, $T$ is not $\Fcal$-radical. Hence no subgroup $P\subgroupeq T$ can be both $\Fcal$-centric and $\Fcal$-radical.

A useful general tool for the remaining cases is given by Proposition~\ref{proposition:FiniteCentricRadicals}, which tells us that
the image of the bullet construction includes all subgroups of $S$ that are both $\Fcal$-centric and $\Fcal$-radical. Hence if
$P^\bullet$ contains $P$ as a proper subgroup, then $P$ cannot be both $\Fcal$-centric and $\Fcal$-radical.

We have already eliminated subgroups of~$T$ as possibilities for subgroups that are both $\Fcal$-centric and $\Fcal$-radical. So suppose $P$ is a proper subgroup of~$S$ that is not contained in~$T$.
Then $P\cap T$ is a proper subgroup of~$T$, and we need to know how big it can be.
If $\integers/8\subgroupeq  P\cap T$, then
by Lemma~\ref{lemma: SU2 bullets} we know $T\subgroupeq P^\bullet$, and by construction
$P\subgroupeq P^{\bullet}$ as well.
Because $P$ is not contained in~$T$, we conclude that
$P^{\bullet}=S$. Since $P$ is a proper subgroup of $S$ by assumption,
we conclude that $P$ is properly contained in~$P^\bullet$, so $P$ is not $\Fcal$-centric and $\Fcal$-radical.
Further, $P\cap T=0$ is not possible, because all elements of $S$ outside the torus have squares that are in the torus.

The only remaining possibilities are for $\Fcal$-centric and $\Fcal$-radical subgroups $P$ to have $P\cap T=\integers/2$ or $P\cap T=\integers/4$, which are
the possibilities \itemref{item: nonsurviving option} and
\itemref{item: Q extension} in the statement. Note that all elements of $S$ that are not contained in~$T$ are conjugate in~$S$, because $[S:T]=2$; as a result,
there is only one conjugacy class of each type of subgroup.
\end{proof}

We can readily eliminate option~\itemref{item: nonsurviving option} of Lemma~\ref{lemma: radical possibilities}.

\begin{lemma} \label{lem:Z/4-not-radical}
The subgroup $P=\langle\,\imatrix\,\rangle\cong\integers/4$ is not $\Fcal$-radical.
\end{lemma}

\begin{proof}
Abstractly, there is exactly one nontrivial automorphism of~$P$, namely the one that exchanges $\imatrix$ with~$\imatrix^{3}$. This nontrivial automorphism is realized by conjugation by~$\jmatrix$.
Hence $\Out_{\Fcal}P = \Aut_{\Fcal}P=\integers/2$, and $P$ is not $\Fcal$-radical.
\end{proof}

\subsection*{The quaternionic subgroup of~$S$.}

We still need to check if the group ~\itemref{item: Q extension} in Lemma~\ref{lemma: radical possibilities} is an $\Fcal$-centric and $\Fcal$-radical subgroup of~$S$. We call this group $Q\subset\SU(2)$, defined by the following representation:

\footnotesize
\begin{equation}    \label{diagram: Q matrices}
\begin{gathered}
\begin{array}{llll}
I=\twobytwo{1}{0}{0}{1},
&\imatrix=\twobytwo{0}{1}{-1}{0},
&\jmatrix=\twobytwo{i}{0}{0}{-i},
&\kmatrix=\twobytwo{0}{-i}{-i}{0},\\
\\
-I=\twobytwo{-1}{0}{0}{-1},
&\imatrixbar=\twobytwo{0}{-1}{1}{0},
&\jmatrixbar=\twobytwo{-i}{0}{0}{i},
&\kmatrixbar=\twobytwo{0}{i}{i}{0} .
\end{array}
\end{gathered}
\end{equation}
\normalsize
The elements multiply as the quaternions do: $\imatrix\jmatrix=\kmatrix$, $\jmatrix\kmatrix=\imatrix$, and $\kmatrix\imatrix=\jmatrix$, as well as
$\imatrix^2=\jmatrix^2=\kmatrix^2=-I$.
Straightforward computation establishes the following lemma.

\begin{lemma}\hfill
\begin{enumerate}[(i)]
\item The center of $Q$ is $\{I,-I\}$.
\item There are six elements of order four: $\imatrix, \jmatrix, \kmatrix, \imatrixbar, \jmatrixbar, \kmatrixbar$.
\item The inner automorphism group of $Q$ is $\integers/2\times\integers/2$:
\[
\begin{array}{lll}
c_{\imatrix}=c_{\imatrixbar}\qquad
& c_{\jmatrix}=c_{\jmatrixbar}\qquad
&c_{\kmatrix}=c_{\kmatrixbar}\\
\imatrix\mapsto \imatrix
& \imatrix\mapsto \imatrix^{-1}
& \imatrix\mapsto \imatrix^{-1} \\
\jmatrix\mapsto  \jmatrix^{-1}
& \jmatrix\mapsto \jmatrix
& \jmatrix\mapsto \jmatrix^{-1} \\
\kmatrix\mapsto \kmatrix^{-1}
& \kmatrix\mapsto \kmatrix^{-1}
& \kmatrix\mapsto \kmatrix.
\end{array}
\]
\end{enumerate}
\end{lemma}

In order to calculate $\Aut_\Fcal(Q)$,
we begin with the abstract automorphisms of~$Q$ (i.e. not assuming that they arise from conjugation). Any automorphism must fix the center, $\{I,-I\}$, and can permute the six elements of order~$4$.

\begin{lemma}              \label{lemma: automorphism group}
The abstract automorphism group $\Aut(Q)$ is isomorphic to~$\Sigma_{4}$. The abstract outer automorphism group
of $Q$ is isomorphic to~$\Sigma_{3}$.
\end{lemma}

\begin{proof}
Any automorphism of $Q$ must fix the center,
and then permute the remaining elements of $Q$ while respecting the multiplication. The automorphism group can be realized as the orientation-preserving symmetries of the set of coordinate axes of~$\reals^3$, since those symmetries give permutations of the coordinate vectors $\pm\vec{\textit{\i}}$, $\pm\vec{\textit{\j}}$, $\pm\vec{\textit{k}}$
that respect the multiplicative relations between the elements of $Q$ with corresponding names. In turn we can see that such symmetries also correspond to symmetries of the cube with vertices $(\pm 1, \pm 1, \pm 1)$.
Finally, symmetries of the cube are in one-to-one correspondence with the permutations of the four interior, cross-cube diagonals, establishing that the abstract automorphism group of $Q$ is isomorphic to~$\Sigma_{4}$.

Since there are four inner automorphisms of~$Q$ and the full automorphism group has order~$24$, the group of outer automorphisms of $Q$ has order~$6$. In addition, the outer automorphisms act on the set of subgroups of $Q$ of order~$4$, namely $\left\{\left\langle\strut\imatrix\right\rangle,
\left\langle\strut\jmatrix\right\rangle,
\left\langle\strut\kmatrix\right\rangle\right\}$.
Every permutation can be achieved by an automorphism of~$Q$, and the lemma follows.
\end{proof}

We are considering the particular representation of the quaternions in~\eqref{diagram: Q matrices}, which we denote by~$\rho:Q\rightarrow\SU(2)$, and we want to determine if it is $\Fcal$-centric and $\Fcal$-radical.
First, we relate abstract automorphisms of $Q$ to the fusion system, $\Aut_\Fcal(Q)\subset \Aut(Q)$.

\begin{lemma} \label{lemma: Q automorphisms realized}
All abstract automorphisms of $Q$ can be realized as conjugations of~$\rho$ by elements of~$\SU(2)$. In particular, $\Aut_\Fcal(Q)\cong\Aut(Q)\cong\Sigma_{4}$,
and $\Out_\Fcal(Q)\cong\Out(Q)\cong\Sigma_{3}$.
\end{lemma}

\begin{proof}
Let $\character{\rho}$ denote the character of~$\rho$, that is, $\character{\rho}(q)$ is the trace of the matrix assigned to the element $q\in Q$.
We observe from the matrix presentation of~$Q$ in~\eqref{diagram: Q matrices} that
$\character{\rho}(I)=2=-\character{\rho}(-I)$,
and $\character{\rho}(q)=0$ if $q\notin Z(Q)$.

Let $f\colon Q\rightarrow Q$ be an automorphism of~$Q$. Then $f$ fixes $Z(Q)$ and permutes the other six elements. Thus
$\character{\rho\circ f}=\character{\rho}$. Since we are considering complex representations, by character theory we know that $\character{\rho\circ f}=\character{\rho}$ implies that
$\rho$ and $\rho\circ f$ are conjugate in the unitary group~$\U(2)$. By dividing the conjugating matrix by the square root of its determinant, we obtain a matrix in $\SU(2)$ that conjugates
$\rho$ to $f\circ \rho$, as required.
\end{proof}

At this point, we have all of the ingredients necessary to determine the subgroups in the fusion system $\Fcal=\Fcal_{S}(\SU(2))$ that are both $\Fcal$-centric and $\Fcal$-radical.

\begin{proposition}    \label{proposition: two relevant groups}
There are exactly two $\Fcal$-conjugacy classes of subgroups that are both $\Fcal$-centric and $\Fcal$-radical, namely $S$ itself and the $\Fcal$-conjugacy class of~$Q$.
\end{proposition}

\begin{proof}
By Lemmas~\ref{lemma: radical possibilities} and~\ref{lem:Z/4-not-radical}, the groups
$S$ and the $S$-conjugates of $Q$ are the only subgroups of $S$ that could be both $\Fcal$-centric and $\Fcal$-radical. The group $S$ itself is $\Fcal$-centric tautologically, and $\Fcal$-radical by Remark~\ref{remark: S is radical}, so consider~$Q$. Straightforward computation shows that if $A\in \SU(2)$ commutes with the elements in~$Q$, then $A$ must be $\Id$ or~$-\Id$.
Hence
\begin{equation}   \label{eq: centralizer-center}
C_{\SU(2)}Q=\{\pm I\}=Z(Q).
\end{equation}
Since $Q$ contains all elements of $\SU(2)$ that centralize it, then certainly $Q$ contains all elements of $S$ that centralize it, so $Q$ is $\Fcal$-centric. Secondly, to see that $Q$ is $\Fcal$-radical, notice that from Lemma~\ref{lemma: automorphism group} and Lemma~\ref{lemma: Q automorphisms realized}, we have $\Out_{\Fcal}Q\cong\Out(Q)\cong\Sigma_{3}$, which has no normal \twodash subgroup.
\end{proof}

In order to compute~$\pcomplete{\BSU(2)}$, we know from Theorem~\ref{theorem:Lpcom=BGpcom} that we should determine the centric linking system associated to the fusion category.
Further,
Proposition~\ref{proposition:Hfamilies} tells us that we only need to know the full subcategory $\Lcal^{cr}$ of the linking system that has just the $\Fcal$-centric, $\Fcal$-radical subgroups as its objects.
Lastly, we know from Proposition~\ref{proposition: two relevant groups} that there are only two isomorphism (i.e. $\Fcal$-conjugacy) classes of such groups,
namely $S$ itself and~$Q$.

To compute the linking system, we actually find ourselves in the situation of Corollary~\ref{corollary:T=BG}, and we can use the transporter category (Definition~\ref{definition:transporter}).
This is because $C_{\SU(2)}Q=\{\pm I\}$ is a \twodash group
(see~\eqref{eq: centralizer-center}), and likewise
\begin{equation}   \label{eq: centralizer-S-center}
C_{\SU(2)}S=\{\pm I\}=Z(S),
\end{equation}
because $S\supergroup Q$. Hence computing
$N_{\SU(2)}S$, $N_{\SU(2)}Q$, and $N_{\SU(2)}(Q,S)$ will give us
an explicit description of~$\Lcal^{cr}$.

\begin{lemma}     \label{lemma: normalizers of S and Q}
\hfill
\begin{enumerate}[(i)]
\item The normalizer of $S$ in $\SU(2)$ is~$S$.
\item The normalizer of $Q$ in $\SU(2)$ is~$O_{48}$, the binary octahedral group of order~$48$.
\end{enumerate}
\end{lemma}

\begin{proof}
Suppose that $n\in N_{\SU(2)}S$. Then $n$ normalizes the identity component
of~$S$, which is~$T$. Hence $n\in N_{\SU(2)}T=S$. Because $S$ has to be
a discrete \twodash toral group, in principle, $S$ could be strictly contained in $N_{\SU(2)}T$; that does at the prime~$2$
because $N_{\SU(2)}T$ is already \twodash toral.

Next, consider the normalizer of $Q$ in~$\SU(2)$.
All abstract automorphisms of $Q$ are realized by conjugation in~$\SU(2)$
(Lemma~\ref{lemma: automorphism group}). Further, we have a central extension
\[
1\rightarrow \Cen_{\SU(2)}Q
\rightarrow N_{\SU(2)}Q
\rightarrow \Aut_{\Fcal}(Q)
\rightarrow 1,
\]
and by ~\eqref{eq: centralizer-center} $\Cen_{\SU(2)}Q=Z(Q)=\{\pm I\}$, while by Lemma~\ref{lemma: Q automorphisms realized},
$\Aut_{\Fcal}(Q)\cong\Sigma_{4}$. So the normalizer is a finite group of order~$48$.
We assert that this group is the ``binary octahedral group,"~$O_{48}$; by definition, this is the group of unit quaternions $q\in S^3\cong \SU(2)$ that permute the vertices
$\{\pm\vec{\textit{\i}}, \pm\vec{\textit{\j}}, \pm\vec{\textit{k}}\}$
of an octahedron, acting by conjugation.
The subgroup $Q$ is a normal subgroup of $O_{48}$, so $O_{48}\subgroupeq N_{SU(2)}(Q)$ and since they have the same order, we find that $O_{48}=N_{SU(2)}(Q)$.
\end{proof}

We summarize what we have computed so far. The isomorphism classes of objects of
$\Tcal^{cr}$ are represented by $S$ and~$Q$. Morphisms for $P,Q\in \Obj(\Tcal^{cr})$
are given by  $\Hom_{\Tcal^{cr}}(P,Q)=N_{\SU(2)}(P,Q)$, so in particular
$\Aut_{\Tcal^{cr}}(S)=S$ and $\Aut_{\Tcal^{cr}}(Q)=O_{48}$.
So we have
\begin{equation}    \label{eq: SU linking system}
\begin{gathered}
 \xymatrix{
Q\ar@(dl,lu)^{O_{48}}\ar[r]^-{N} & S\ar@(dr,ru)_{S}
}
\end{gathered}
\end{equation}
where the set of morphisms from $Q$ to~$S$ is $N=N_{\SU(2)}(Q,S)$.
To complete our understanding of the transporter category,
we would need to compute~$N$.
However, the proof we give of the main result in Section~\ref{section: SU(2) SO(3)}
takes a different
tack,
for which we need
$N_{\SU(2)}(Q)\cap N_{\SU(2)}(S)$. That is, we want to know what elements of
$N_{\SU(2)}S=S$ normalize~$Q$, i.e. $N_{S}Q$.
Let $Q_{16}$ denote the generalized quaternion group of order~$16$, which has the presentation
\[
\left\langle\largestrut a, b\suchthat{a^8=1,\, a^4=b^2,\, b^{-1}ab=a^{-1}}
\right\rangle.
\]

\begin{lemma}   \label{lemma: normalizer of Q}
$N_{S}(Q)\cong Q_{16}$.
\end{lemma}

\begin{proof}
Since $\imatrix\in Q$, and $\imatrix$ together with $T$ generates~$S$, there is a short exact sequence
\begin{equation}     \label{eq: normalizer ses}
1 \rightarrow N_{S}(Q)\cap T
  \rightarrow N_{S}(Q)
  \rightarrow \integers/2
  \rightarrow 1,
\end{equation}
where the quotient is generated by~$\imatrix$.
To determine $N_{S}(Q)\cap T$, we identify elements of $T$  with complex exponentials, and we can write
\[
Q=\left\{e^{i\left(n\frac{\pi}{2}\right)},
\,\imatrix e^{i\left(n\frac{\pi}{2}\right)}\suchthat{\largestrut n=0,1,2,3 }\right\}.
\]
If $e^{ix}\in T$ normalizes~$Q$, it must
conjugate $\imatrix$ back into~$Q$. Conjugation by $e^{ix}$ fixes
$Q\cap T
=\{e^{i\left(n\frac{\pi}{2}\right)}\suchthat{ n=0,1,2,3} \}$,
so there must be a value of $n$ such that
\begin{align*}
e^{-ix}\,\imatrix\, e^{ix}
    &= \imatrix\, e^{i\left(n\frac{\pi}{2}\right)}.\\
\intertext{Since $\imatrix$ acts on the torus by complex conjugation, we get}
e^{i(2x)}&=e^{i\left(n\frac{\pi}{2}\right)}.
\end{align*}
Hence
$N_{S}(Q)\cap T\cong\langle e^{\pi i/4}\rangle$.
Let $a=e^{i\pi/4}\in T$ and $b=\imatrix$, and observe that $a$ has order~$8$, that $a^4=b^2$, and
\[
b^{-1}ab=\imatrix^{-1}e^{i\pi/4}\imatrix = e^{-i\pi/4}=a^{-1},
\]
again because $\imatrix$ acts on the torus by complex conjugation.
The result now follows from~\eqref{eq: normalizer ses} because we have accounted for all the elements and exhibited an appropriate presentation.
\end{proof}


\section{Mod $2$ decomposition of $\BSU(2)$ and $\BSO(3)$}
\label{section: SU(2) SO(3)}

In this section we will prove a discrete version of the Dwyer-Miller-Wilkerson decompositions of $\BSU(2)$
and~$\BSO(3)$ at the prime~$2$ by using the discrete \twodash local information obtained in Section~\ref{section:2local}. We follow the philosophy of Dwyer's normalizer decomposition for classifying spaces of finite groups \cite{Dwyer-Homology}, together with that of Libman's proof of the existence of the normalizer decomposition for classifying spaces of saturated fusion systems over finite \pdash groups \cite{Libman-normalizer}.

We continue with the notation of Section~\ref{section:2local}.
We write $\Fcal$ for the fusion system associated to the compact Lie group $\SU(2)$ with maximal discrete \twodash toral subgroup~$S$.
Let $\Tcal^{cr}$ denote the $\Fcal$-centric and $\Fcal$-radical subcategory of the transporter category associated to $\SU(2)$. The isomorphism classes of objects, together with their automorphism groups, were explicitly computed in Section~\ref{section:2local}. Representatives of the classes
of objects are given by $S$ and~$Q$, with $\Aut_{\Tcal^{cr}}(S)=S$ and $\Aut_{\Tcal^{cr}}(Q)=O_{48}$ (Lemma~\ref{lemma: normalizers of S and Q}).

The category $\Tcal^{cr}$ is an EI-category; that is, all endomorphisms are isomorphisms. We apply the techniques developed by \Slominska~\cite{S} in the homotopy theory of EI-categories, and implemented by Libman \cite{Libman-normalizer},
to describe the nerve of $\Tcal^{cr}$ as a homotopy colimit indexed on a poset category. Evaluating this homotopy colimit leads to a decomposition of $\BSU(2)$.

The first step is to consider the subdivision category of a ``heighted category." Suppose that $\Ccal$ is an EI-category equipped with a height function $h\colon \Obj(\Ccal)\rightarrow \naturals$. It is assumed that morphisms of $\Ccal$ do not decrease height, and that isomorphisms are exactly the morphisms that preserve height. The \defining{subdivision category of~$\Ccal$}, denoted~$s(\Ccal)$, is a category with objects $\Cbold$ given by chains of composable morphisms in $\Ccal$ that strictly increase height,
\[
\Cbold=(c_0\to \cdots \to c_n).
\]
A morphism from $\Cbold\rightarrow \Cbold'$ is given by compatible isomorphisms from $\Cbold'$ to a subsequence in $\Cbold$
(see \cite[Proposition 1.3]{S} or \cite[Definition 4.1]{Libman-normalizer}). In particular, an automorphism of $\Cbold$ is an automorphism of $c_n$ that restricts to automorphisms of $c_i$ for every $0\leq i <n$.

\begin{proposition}\cite[Proposition~1.5]{S}
\label{proposition:subdivision}
Let $\Ccal$ be an EI-category equipped with a height function.
The functor $s(\Ccal)\rightarrow \Ccal$
given by $\Cbold\mapsto c_0$ is right cofinal, and in particular
induces a homotopy equivalence
$\realize{s(\Ccal)}\simeq\realize{\Ccal}$.
\end{proposition}

\begin{example}\label{example:subdivisionofL}
The category $\Tcal^{cr}$, pictured in~\eqref{eq: SU linking system}, is an EI-category with a height function given by $h(Q)=0$ and $h(S)=1$. We assert that its subdivision category $s(\Tcal^{cr})$ has the following skeletal subcategory:
\begin{equation} \label{eq: SU EI system}
\begin{gathered}
 \xymatrix{
\{Q\} \POS!L(0.7)\ar@(ul,dl)_{O_{48}}
 & \{Q\subgroup S\}\POS!U(.1) \ar@(ul,ur)^{Q_{16}}
 \ar[r]^-{S}
 \ar[l]_-{O_{48}}
 & \{S\} \POS!R(.7) \ar@(ur,dr)^{S}
 }
 \end{gathered}
\end{equation}
where the labels on the arrows indicate the morphism sets.
The automorphisms of $\{Q\}$ and $\{S\}$ come from Lemma~\ref{lemma: normalizers of S and Q}.
An automorphism of the chain $\{ Q\subseteq S \}$ is an automorphism of $S$ that
restricts to an automorphism of $Q$; that is, an element of $N_{\SU(2)}(Q)\cap
N_{\SU(2)}(S)$.
Using Lemma~\ref{lemma: normalizers of S and Q} and Lemma~\ref{lemma: normalizer
of Q}, we obtain
\begin{align*}
\Aut_{s(\Tcal^{cr})}(Q\subgroup S)
     &=N_{\SU(2)}(Q)\cap N_{\SU(2)}(S)\\
     &=N_{S}(Q)\\
     &=Q_{16}.
\end{align*}
Morphisms $\{Q\subgroup S\}\rightarrow\{S\}$ are just diagrams
\[
\xymatrix{
&S  \ar[d]\\
Q \ar[r]& S,
}
\]
and these are determined by morphisms $S\to S$ in $\Tcal^{cr}$.
Therefore we find that
$\Hom_{\Tcal^{cr}}(\{ Q\subgroup S \}, S)
     \cong \Aut_{\Tcal^{cr}}(S)\cong S$.
Likewise,
$\Hom_{\Tcal^{cr}}(\{ Q\subgroup S \}, \{ Q \})
     \cong \Hom_{\Tcal^{cr}}(Q)\cong O_{48}$.
Lastly, each of these two morphism sets has a free transitive action of the automorphism group of its target object.
\end{example}

The theorem below (Theorem~\ref{theorem:DMWSU(2)}) is the first part of our main result. We establish the decomposition of $\pcomplete{\BSU(2)}$ given in~\cite[Theorem~4.1]{DMW1}
as an instance of a
normalizer decomposition (in the sense of \cite{Dwyer-Homology}) with respect $\Fcal$-centric $\Fcal$-radical discrete \pdash toral subgroups, applied to the transporter category~$\Tcal^{cr}$.
Earlier work of Libman \cite{Libman-normalizer} established
the existence of a normalizer decomposition for saturated fusion systems over finite \pdash groups. In our case, we deal with a fusion system associated to an infinite, discrete \twodash toral group $S$ instead of a finite group.
Libman's argument uses the centric linking system, whereas we focus on the transporter category~$\Tcal^{cr}$.

\begin{theorem}\label{theorem:DMWSU(2)}
Let $S\subgroupeq  \SU(2)$ be the maximal  discrete \twodash  toral subgroup generated by $\ZpinfinitySpecific{2}$ and \small
$\imatrix=\twobytwo{0}{1}{-1}{0}$\normalsize$\in \SU(2)$.
Let $Q_{16}\subgroupeq O_{48}\subgroupeq \SU(2)$ be, respectively, the
quaternionic subgroup of order~$16$ and the binary octahedral subgroup of order~$48$. Let $X$ be the homotopy pushout
\begin{equation}\label{eq:BSO(3)-pushout}
\begin{gathered}
\xymatrix{
B\,Q_{16}
   \ar[d]\ar[r]
& BS
   \ar[d]\\
B\,O_{48}
   \ar[r]
& X,
}
\end{gathered}
\end{equation}
where the maps are induced by inclusions of subgroups. Then the  \twodash  completion of the induced map $X\rightarrow \BSU(2)$ is a homotopy equivalence.
\end{theorem}

\begin{proof}
We want to apply Corollary~\ref{corollary:T=BG} in order to use the transporter category instead of the linking category to find~$\pcomplete{\BSU(2)}$.
From \eqref{eq: centralizer-center}
and~\eqref{eq: centralizer-S-center},
we have the required condition for~$Q$ and~$S$.
Hence $\Tcal^{cr}\rightarrow \Bcal\SU(2)$ induces an  equivalence $\pcomplete{\realize{\Tcal^{cr}}}\simeq \pcomplete{\BSU(2)}$
by Corollary~\ref{corollary:T=BG}. We also have an homotopy equivalence $\realize{\Tcal^{cr}}\simeq \realize{s(\Tcal^{cr})}$ by Proposition \ref{proposition:subdivision}. We focus then on the subdivision category~$s(\Tcal^{cr})$, to describe its nerve as a homotopy pushout.

The subdivision category
$s(\Tcal^{cr})$ is pictured in~\eqref{eq: SU EI system}.
The underlying category for the pushout will be given by considering isomorphism classes of objects in $s(\Tcal^{cr})$.
Let $\sbar(\Tcal^{cr})$ be the \emph{poset} of isomorphism classes of objects in $s(\Tcal^{cr})$ given by
 \[
 \{S\}\longleftarrow \{Q\subgroup S\} \longrightarrow \{Q\}.
 \]
Consider the functor $\Bcal\colon \sbar{\Tcal^{cr}}\rightarrow \Cat$ given by the one object category of the corresponding automorphism group in $s(\Tcal^{cr})$. That is,
\begin{align}     \label{eq: what are the spaces}
\Bcal(\{S\})&=\Bcal\Aut_{s(\Tcal^{cr})}(\{S\})=\Bcal S
     \notag\\
\Bcal(\{Q\})&=\Bcal\Aut_{s(\Tcal^{cr})}(\{Q\})=\Bcal O_{24}\\
\Bcal(\{Q\subgroup S\})&=\Bcal\Aut_{s(\Tcal^{cr})}(\{Q\subgroup S\})=\Bcal Q_{16},
     \notag
\end{align}
with functors induced by the inclusion of subgroups.
The Grothendieck construction $\Gr(\Bcal)$ is isomorphic
to~\eqref{eq: SU EI system}. See, for example, \cite[2.9]{Dwyer-Homology}. The objects of $\Gr(\Bcal)$ are also $\Obj(\sbar(\Tcal^{cr}))$ and morphisms are given by a inclusion or identity followed by an automorphism of the target. Hence Thomason's theorem \cite[2.9]{Dwyer-Homology} tells us that the nerve of
\eqref{eq: SU EI system} is homotopy equivalent to the homotopy colimit of~\eqref{eq:BSO(3)-pushout}, as required.
\end{proof}

\medskip

Lastly, we turn our attention to~$\BSO(3)$. In \cite[Corollary 4.2]{DMW1}, the mod~$2$ decomposition for $\BSO(3)$ is derived from the
mod~$2$ decomposition of $\BSU(2)$ using the fibration sequence associated to the double cover $\SU(2)\rightarrow\SO(3)$. We show that their decomposition for $\pcomplete{\BSO(3)}$ corresponds to the classifying space of a fusion system
in a way precisely analogous to the $\BSU(2)$ case.

The Lie group $\SO(3)$ is the quotient of $\SU(2)$ by its center $\integers/2=\{\pm I\}$, contained in the maximal torus~$S^1$.
In what follows, if $H$ is a subgroup of~$\SU(2)$, we write $\Hbar$ for its image in~$\SO(3)$. We sketch the argument that all of the structures we need are preserved by the quotient map.
First, the maximal torus $\Tboldbar$ of $\SO(3)$ is the quotient of the maximal torus $\Tbold$ of~$\SU(2)$. Likewise, we can see that
$N_{\SO(3)}\Tboldbar\cong (N_{\SU(3)}\Tbold)/\{\pm I\}$,
because $\{\pm I\}$ is central and contained in the torus.

The same statements are true for the discrete structures. That is, if $T$ and $S$ denote a maximal discrete \twodash torus and maximal \twodash toral subgroup of~$\SU(2)$, then $\Tbar$ and $\Sbar$ play the same roles in~$\SO(3)$.
This time, however, the extension
$\Tbar \hookrightarrow N_{\SO(3)}(\Tbar) \epi \integers/2$ is split. The splitting occurs because $\iboldbar$ (representing the generator of $\integers/2$) actually has order $2$ in~$\SO(3)$, as opposed to $\ibold$ which had order~$4$ in~$\SU(2)$ (see \eqref{eq: normalizer in SU(2)} and~\eqref{eq: ses for S}). Thus the maximal discrete \twodash toral subgroup $\Sbar$ is isomorphic to the semi-direct product $(\integers/2^\infty)\rtimes \integers/2$.

We need the following notation. Let $V\subgroupeq \Sbar$ be the subgroup $\integers/2\times\integers/2$, generated by $\integers/2\subgroupeq \Tbar$ and by~$\iboldbar$. Also, let $\Fcalbar$ and $\Tcalbar$ be the fusion and transporter system, respectively, for~$\SO(3)$ (but we do not mean to imply by the notation that these are actually quotients).

\begin{proposition}      \label{proposition: SO(3) transporter category}
The subdivision category $s(\Tcalbar^{cr})$ of $\Tcalbar^{cr}$ has the following skeletal subcategory:
\vspace{-1em}
\begin{equation*}
\begin{gathered}
 \xymatrix{
\{V\} \POS!L(0.7)\ar@(ul,dl)_{\overline{O_{48}}}
 & \{V\subgroupeq \Sbar\}\POS!U(.1) \ar@(ul,ur)^{\overline{Q_{16}}}
 \ar[r]^-{\Sbar}
 \ar[l]_-{\overline{O_{48}}}
 & \{\Sbar\} \POS!R(.7) \ar@(ur,dr)^{\Sbar}
 }
 \end{gathered}
\end{equation*}
\end{proposition}

\begin{proof}
The calculation is very similar to that of Section~\ref{section:2local}. First we need
the collection of $\Fcalbar$-centric, $\Fcalbar$-radical subgroups of $\Sbar$ and their normalizers.
As was the case for~$\SU(2)$ (Lemma~\ref{lemma: radical possibilities}), no proper subgroup of $\Tbar$ is $\Fcalbar$-centric, and $\Tbar$ itself is not $\Fcalbar$-radical.

Likewise an identical argument says that
if $P$ is a proper $\Fcalbar$-centric, $\Fcalbar$-radical subgroup of~$\Sbar$, then $P$ does not contain $\integers/8\subgroupeq \Tbar$.
In principle, instead of two possibilities remaining
(as in Lemma~\ref{lemma: radical possibilities}), there are three, because $P\,\cap\,\Tbar=0$ is now possible (for $P=\langle \ibold\rangle$). However, $\langle \ibold\rangle$ is not $\Fcalbar$-centric, since it is centralized by the element of order~$2$ in~$\Tbar$.

We are left with two possible extensions of $\langle \ibold\rangle$ by a subgroup of~$\Tbar$. The first is trivial, and the second is the dihedral group of symmetries of the square:
\[
1\longrightarrow \integers/2
 \longrightarrow V
 \longrightarrow \integers/2
 \longrightarrow 1,
\]
\[
1\longrightarrow \integers/4
 \longrightarrow D_{4}
 \longrightarrow \integers/2
 \longrightarrow 1.
\]
However, the inner automorphism group $\Aut_{D_{4}}(D_{4})\cong D_{4}/Z(D_{4})$ has order~$4$, while the full abstract automorphism group is isomorphic to~$D_{4}$ and has order~$8$.
Further, all abstract automorphisms of $D_{4}$ are realized by conjugations in~$\SO(3)$, as shown by thinking of $D_{4}$ as symmetries of a square. Therefore $\Out_{\Fcalbar}(D_{4})\cong\integers/2$, so
$D_{4}$ is not $\Fcalbar$-radical.

On the other hand, the abstract automorphism group of~$V$ is~$\Sigma_{3}$ and is also fully realized in~$\SO(3)$, so $V$ is $\Fcalbar$-radical. By direct computation with matrices, it is also easily verified that $C_{\SO(3)}(V)=V$, and so $V$ is indeed $\Fcalbar$-centric.

Note that $V$ is the image of $Q$ under the quotient map~$\SU(2)\rightarrow\SO(3)$. Further, taking a quotient group by the center preserves normalizers of subgroups that contain the center. That is, $N_{\SO(3)}(V)\cong N_{\SU(2)}(Q)/Z(\SU(2))=\overline{O_{48}}$. Likewise, $N_{\SO(3)}(\Sbar)=\Sbar$,
and $N_{\SO(3)}(\Sbar)\cap N_{\SO(3)}(V)=\overline{Q_{16}}$.
\end{proof}

For the statement of Theorem~\ref{theorem:DMWSO(3)}, observe that in fact the quotient $\overline{Q_{16}}$ is the quaternionic group~$Q_{8}$, and likewise the quotient $\overline{O_{48}}$ is the octahedral group~$O_{24}$ (which is isomorphic to~$\Sigma_{4}$).

\begin{theorem}\label{theorem:DMWSO(3)}
Let $\Sbar\subgroupeq  \SO(3)$ be the maximal discrete \twodash  toral subgroup $\ZpinfinitySpecific{2}\rtimes \integers/2$. Let $Y$ be the homotopy pushout
\[\xymatrix{
B\,Q_{8}
   \ar[d]\ar[r]
& B\Sbar
   \ar[d]\\
B\,O_{24}
   \ar[r]
& Y
}
\]
where the maps are induced by inclusions. Then
$\pcomplete{Y}\simeq \pcomplete{\BSO(3)}$.
\end{theorem}

\begin{proof}
Exactly the argument of Theorem~\ref{theorem:DMWSU(2)} applies
(using Proposition~\ref{proposition: SO(3) transporter category}) once we check that
Corollary~\ref{corollary:T=BG} applies in this case as well.
We need to verify that for all $\Fcalbar$-centric $\Fcalbar$-radical subgroups $P\subgroupeq\Sbar$, we know that $C_{\SO(3)}(P)$ is a finite \twodash group and also that $C_{\SO(3)}(P)=Z(P)$. The only cases to verify are $P=\Sbar$ and $P=V$.
We noted in the proof of Proposition~\ref{proposition: SO(3) transporter category} that $C_{\SO(3)}(V)=V$, a finite \twodash group. The $\SO(3)$-centralizer of $\Sbar$ is necessarily smaller, and in fact only $\integers/2\subgroupeq \Tbar$ centralizes~$\Sbar$, so it meets the condition as well. We conclude that Corollary~\ref{corollary:T=BG} allows us to use the transporter category instead of the linking system, and the theorem now follows from Proposition~\ref{proposition: SO(3) transporter category} in exactly the same way that Theorem~\ref{theorem:DMWSU(2)} follows from~\eqref{eq: SU EI system}.
\end{proof}

\bibliographystyle{amsalpha}

\newcommand{\etalchar}[1]{$^{#1}$}
\providecommand{\bysame}{\leavevmode\hbox to3em{\hrulefill}\thinspace}
\providecommand{\MR}{\relax\ifhmode\unskip\space\fi MR }
\providecommand{\MRhref}[2]{%
  \href{http://www.ams.org/mathscinet-getitem?mr=#1}{#2}
}
\providecommand{\href}[2]{#2}

\end{document}